\documentclass[11pt, a4paper]{amsart}
 \usepackage[dvips]{epsfig}
\usepackage{amsmath}
\usepackage{amssymb}
\usepackage{amsfonts}
\usepackage{amsthm}
\usepackage{amsbsy}
\usepackage{amsgen}
\usepackage{amscd}
\usepackage{amsopn}
\usepackage{amstext}
\usepackage{amsxtra}
\usepackage{enumerate}
\usepackage{hyperref}
\usepackage{verbatim}

\usepackage{graphics}
\usepackage{eso-pic}

\numberwithin{equation}{section}

\newtheorem{theorem}{Theorem}[section]
\newtheorem{proposition}[theorem]{Proposition}
\newtheorem{lemma}[theorem]{Lemma}

\theoremstyle{definition}
\newtheorem{definition}[theorem]{Definition}
\newtheorem{notation}[theorem]{Notation}

\newcommand{\Sc}{\mathcal{S}}

\newcommand {\E} {\mathbb{E}}

\newcommand {\R} {\mathbb{R}}
\newcommand {\C} {\mathbb{C}}
\newcommand {\Z} {\mathbb{Z}}

\newcommand {\Cc} {\mathcal{C}}

\newcommand {\Pc} {\mathcal{P}}

\newcommand {\Ac} {\mathcal{A}}

\newcommand {\Tb} {\mathbb{T}}
\newcommand {\var} {\operatorname{Var}}
\newcommand {\sing} {\operatorname{Sing}}

\newcommand{\Ec}{\mathcal{E}}
\newcommand{\Zc}{\mathcal{Z}}

\begin{document}


\title{Points on nodal lines with given direction}
\author{Ze\'ev Rudnick and Igor Wigman}
\address{
School of Mathematical Sciences, Tel Aviv University, Tel Aviv,
Israel} \email{rudnick@post.tau.ac.il}

\address{Department of Mathematics, King's College London, UK}
\email{igor.wigman@kcl.ac.uk}

\date{\today}

\begin{abstract}

We study of the directional distribution function of nodal lines for eigenfunctions of the Laplacian on a planar domain. This quantity counts the number of points  where the normal to the nodal line points in a given direction. We give upper bounds for the flat torus, and compute the expected number for arithmetic random waves.
\end{abstract}

\maketitle

\section{Introduction}

\subsection{Nodal directions}

One of the more intriguing characteristics of a Laplace eigenfunction on a planar domain
is its nodal set. Much progress has been achieved in understanding its length, notably
the work of Donnelly and Fefferman \cite{Donnelly-Fefferman}, and the recent breakthrough by Logunov and Mallinikova \cite{LM, Logunov1, Logunov2}, and several researchers have tried to understand the number of nodal domains  (the connected components of the complement of the nodal set),
starting with Courant's upper bound on that number, see \cite{Bourgain Pleijel} for the latest result.  
In this note, we propose to study a different  quantity, the {\em directional distribution}, measuring an aspect of the curvature of nodal lines.

Let $\Omega$ be a planar domain, with piecewise smooth boundary, and let $f$ be an eigenfunction of the Dirichlet Laplacian, with eigenvalue $E$: $-\Delta f=Ef$.
Given a direction $\zeta\in S^{1} $, let $ N_{\zeta}(f)$ be the number of points $x$ on the nodal line $\{x\in \Omega: f(x)=0\}$ with normal pointing in the direction $\pm \zeta$:
\begin{equation}
\label{eq:Nzeta=f=nablf=zeta def}
N_{\zeta}(f)= \#\left\{x\in\Omega:\: f(x)=0, \frac{\nabla f(x)}{\|\nabla f(x)\|} = \pm \zeta \right\}.
\end{equation}
In particular \eqref{eq:Nzeta=f=nablf=zeta def} requires that $\nabla f(x)\neq 0$, i.e. $x$ is a non-singular point of the nodal line.

In a few separable cases, such as an irrational rectangle, or the disk, one can explicitly compute $N_\zeta(f)$: For the irrational rectangle, the nodal line is a grid and $N_\zeta(f)=0,\infty$, while for the disk the nodal line is a union of diameters and circles, and we find $N_\zeta(f)\ll \sqrt{E}$ except for $O(\sqrt{E})$ choices of $\zeta$, when $N_\zeta(f)=\infty$, see Appendix~\ref{sec:separable}. However, in most cases one cannot explicitly  compute $N_\zeta(f)$.
The following heuristic suggests that generically the order of magnitude of $N_{\zeta}(f)$ is about $E$: We expect a ``typical" eigenfunction to have an order of magnitude of
$E$ nodal domains \cite{Sodin-Nazarov}, and looking at several plots of nodal portraits such as Figure \ref{fig:nod tor} would lead us to believe that many of the nodal domains are ovals, or at least have a controlled geometry, with $O(1)$ points per nodal domain  with normal parallel to any given direction.
Therefore we are led to expect that the total number of points on the nodal line with normal parallel to $\pm \zeta$ should be about
$E$ (if it is finite).


\begin{figure}[ht]
\begin{center}
\includegraphics[height=60mm]{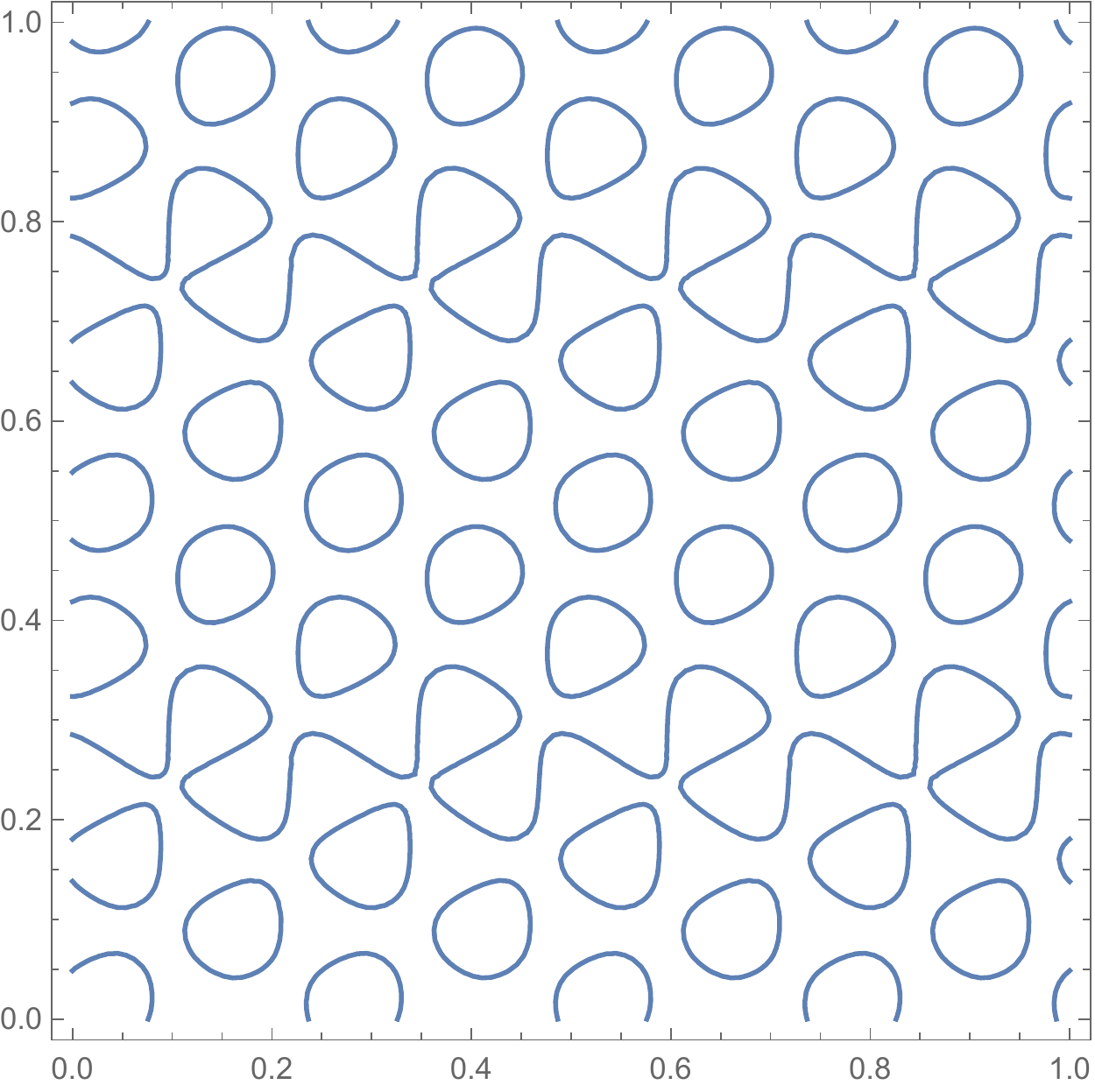}
\caption{The nodal line of the toral eigenfunction $$\sin (2 \pi  (8 x-y))+\sin (2 \pi  (4 x+7 y))+\cos (2 \pi  (4 x-7 y)).$$
A significant proportion of its components are ovals.}
\label{fig:nod tor}
\end{center}
\end{figure}

To try and validate this heuristic, we study $N_\zeta(f)$  on the standard flat torus $\Tb=\R^2/\Z^2$ (equivalently taking $\Omega$ to be the square, and imposing periodic, rather than Dirichlet, boundary conditions), for both random and deterministic eigenfunctions. We prove deterministic upper bounds, and compute the expected value of $N_{\zeta}$ for ``arithmetic random waves" described below.

\subsection{A deterministic upper bound}

We want to establish individual upper bounds on $N_\zeta(f)$.
Strictly speaking, this is not possible, since there are cases where $N_\zeta(f)= \infty$.
For instance, the nodal set of the eigenfunctions $f(x,y)=\sin(2\pi mx)\sin(2\pi ny)$ ($m,n\geq 1$) is a union of straight lines with $N_\zeta(f)=0$ unless $\zeta=\pm(1,0),\pm(0,1)$ in which case $N_\zeta(f)=\infty$.
More generally, one can construct toral eigenfunctions $f$ so that their nodal lines contain a closed geodesic, but also curved components,
see Figure~\ref{fig:nod line geodesic} where we display the eigenfunction
\begin{equation*}
\begin{split}
f(x,y) &=    2\Big(
\sin 8x \sin y + \sin 7x \sin 4y +\sin x \sin 8y +\sin 4x\sin 7y \Big)
\\ &= 4\sin(x)\sin(y)\Big(\cos x+\cos y\Big) h(x,y)
\end{split}
\end{equation*}
where
 \begin{multline}\label{formula for h}
    h(x,y)=   2 \cos (3 x-5 y)-2 \cos (2 x-4 y)-2 \cos (4 x-4   y)+4 \cos (x-3 y)
    +4 \cos (3 x-3 y)     \\  +2 \cos (5 x-3 y)    -4 \cos (2 x-2 y)-2 \cos (4 x-2 y)
        +6   \cos (x-y)+4 \cos (3 x-y)+6 \cos (x+y)
   +4 \cos (3 x+y)
   \\
   -4 \cos (2 x+2 y)   -2 \cos (4 x+2 y)
  +4   \cos (x+3 y)+4 \cos (3 x+3 y)
   +2 \cos (5 x+3 y)
  -2 \cos (2 x+4 y)
  \\  -2 \cos (4 x+4 y)+2 \cos (3 x+5 y)
  -4 \cos (2 x)
     +2 \cos (6 x)-4 \cos (2 y)+2 \cos (6 y)-2 .
     \end{multline}
Theorem~\ref{thm:upper bnd class} below asserts an upper bound for $N_{\zeta}(f)$ with the only exceptions
being when the nodal line contains a closed geodesic. It will follow as a particular case of a structure result on the set
\begin{equation}
\label{eq:Aczeta(f) def}
\Ac_{\zeta}(f) = \left\{ x\in\Omega:\: f(x)=0,  \langle \nabla f(x), \zeta^{\perp} \rangle = 0 \right\}
\end{equation}
of ``nodal directional points", i.e. the set of nodal points where $\nabla f$ is orthogonal to $\zeta^{\perp}$ (thus co-linear to $\zeta$).
Note that, by the definition, in addition to the set on the r.h.s. of \eqref{eq:Nzeta=f=nablf=zeta def}, $\Ac_{\zeta}(f)$ contains all
the singular nodal points of $f^{-1}(0)$, and could also contain certain closed geodesics in direction orthogonal to $\zeta$, as we shall see below.
To state Theorem \ref{thm:upper bnd class} we introduce the (standard) notion of ``height" for a rational vector.

 \begin{figure}[ht]
\begin{center}
\includegraphics[height=60mm]{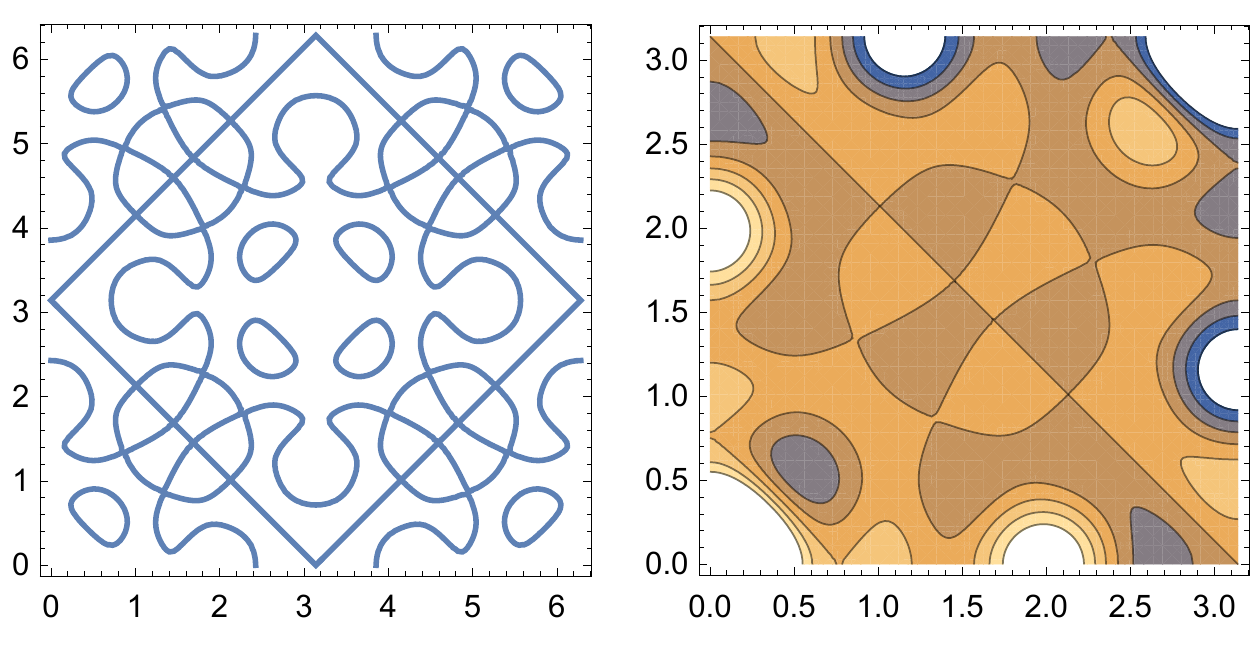}
\caption{Left: nodal set of the eigenfunction
$f(x,y) =  2\Big(
\sin 8x \sin y + \sin 7x \sin 4y +\sin x \sin 8y +\sin 4x\sin 7y \Big)
= 4\sin(x)\sin(y)\Big(\cos x+\cos y\Big) h(x,y)$
for the trigonometric polynomial $h(x,y)$ in \eqref{formula for h},
on the full square $[0,2\pi]\times [0,2\pi]$.
Note the lines $x,y\in\pi \Z$, $x\pm y \in \pi(1+2\Z)$. The scaled function $f(x/2\pi,y/2\pi)$
is a toral eigenfunction.
Right: Contours of $h(x,y)$ on the square $[0,\pi]\times [0,\pi]$.
}
\label{fig:nod line geodesic}
\end{center}
\end{figure}

\begin{notation}[Height of a rational vector]

\begin{enumerate}

\item A {\em rational direction} $\zeta \in \Sc^{1}$ is one which is a multiple of an integer vector. Note that $\zeta$ is rational if and only if the orthogonal direction $\zeta^\perp$ is rational.

\item For a rational vector $\zeta\in\Sc^{1}$ we denote its {\em height} by $h(\zeta) = \max(|k_{1}|,|k_{2}|)$ where $(k_1,k_2)$ is a primitive integer vector (unique up to sign) in the direction of $\zeta$:
  \[
  \zeta = \pm \frac{(k_{1},k_{2})}{\sqrt{k_1^2+k_2^2}}\; .
  \]
  Note that $h(\zeta)=h(\zeta^\perp)$.
\end{enumerate}
\end{notation}

\begin{theorem}
\label{thm:upper bnd class}

Let $\zeta\in\Sc^{1}$ be a direction, and $f$ be a toral eigenfunction: $-\Delta f=Ef$ for some $E>0$.

\begin{enumerate}

\item If $\zeta$ is rational, then the set $\Ac_{\zeta}(f)$ consists of at
most $$\frac{\sqrt{E}}{\pi h(\zeta)}$$
closed geodesics orthogonal to $\zeta$,
at most
$\frac{2}{\pi^{2}} \cdot E$ nonsingular points not lying on the geodesics, and possibly, singular points of the nodal set.

\item If $\zeta$ is not rational, then the set $\Ac_{\zeta}(f)$ consists of at most $\frac{2}{\pi^{2}}\cdot E$ nonsingular points,
and possibly, singular points of the nodal set.

\item In particular, if $\Ac_{\zeta}(f)$ does not contain a closed geodesic, then
\begin{equation*}
N_{\zeta}(f) \le \frac{2}{\pi^2} \cdot E.
\end{equation*}

\end{enumerate}

\end{theorem}

The proof of Theorem \ref{thm:upper bnd class}, given in section \ref{sec:det bnd proof} below,
is sufficiently robust to apply verbatim to the more general family of
trigonometric polynomials on $\Tb^{2}$ of degree $\le \sqrt{E}$.
We note that 
it is possible to construct Laplace eigenfunctions $f$ of arbitrarily high eigenvalues and $\zeta\in\Sc^{1}$ such that
$N_{\zeta}(f)=0$ vanishes, so that a general lower bound for $N_{\zeta}(f)$ cannot exist. For example,
$$f(x,y)=2 \cos(2\pi\cdot m x) + \cos(2\pi\cdot m y)$$
has eigenvalue $E=4\pi^2m^2$ and satisfies $N_{\zeta}(f)=0$ for $\zeta = e^{i\theta}$ with $\theta$ near $\pi/2$, see
Figure \ref{fig:vertical nodal}.

\begin{figure}[ht]
\centering
\includegraphics[height=60mm]{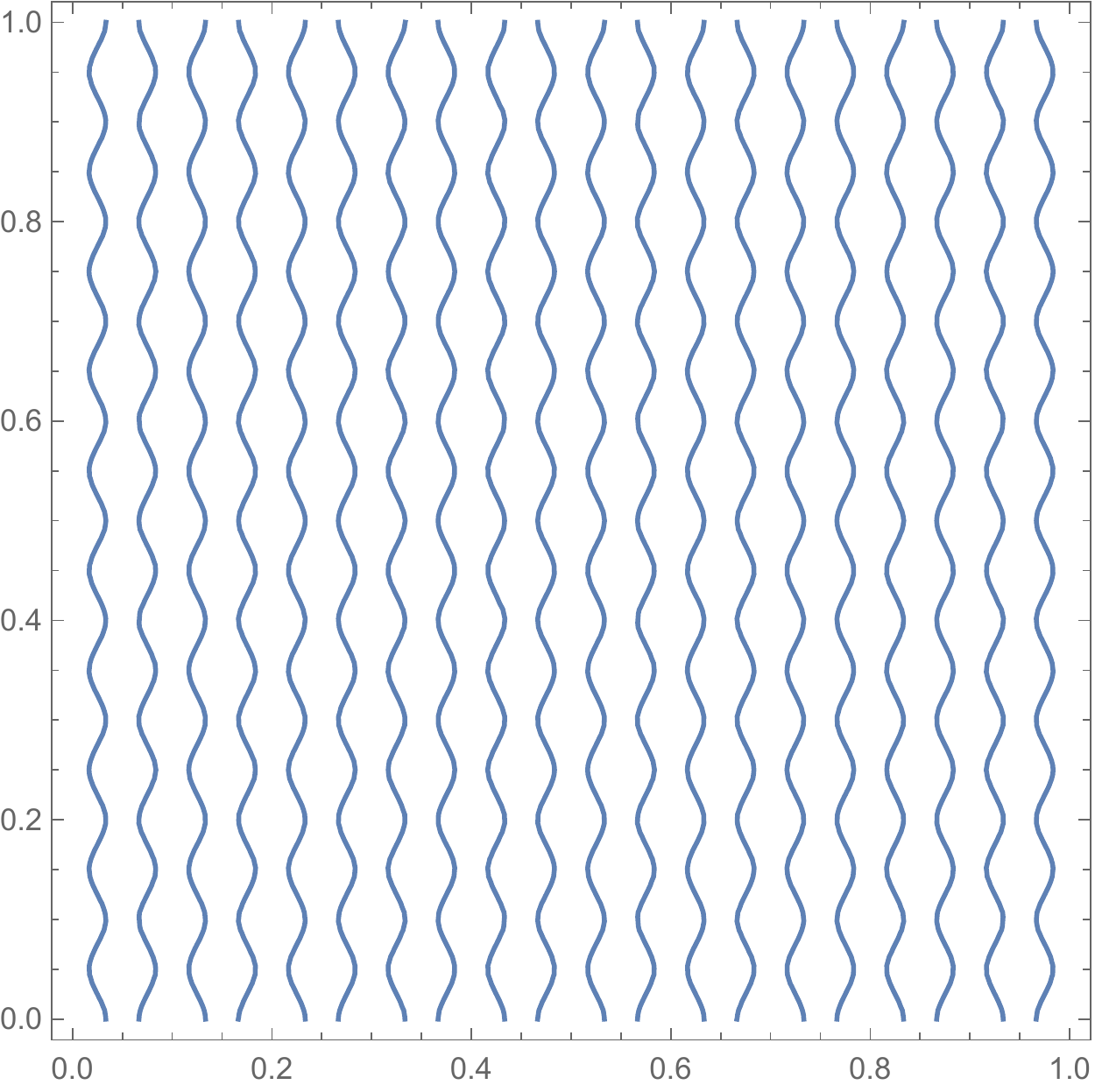}
\caption{The nodal line of $f(x,y)=2\cos(2\pi \cdot 10x)+ \cos(2\pi\cdot 10 y)$. For the choice
$\zeta = e^{i\pi/2}$ we have $N_{\zeta}(f)=0$.}
\label{fig:vertical nodal}
\end{figure}

\subsection{Expected number for arithmetic random waves}

A better understanding of several properties of nodal lines is obtained if one studies {\em random} eigenfunctions.
In 1962, Swerling \cite{Swerling} studied statistical properties of contour lines of a general class of planar Gaussian processes, and gave a non-rigorous computation of the expected value of $N_\zeta$ for general contour lines,  using the result to bound the number of closed connected components of contour lines.
 We will compute the expected value of $N_\zeta$ for
``arithmetic random waves" \cite{ORW2007, RW2008}.
These  are random eigenfunctions on the torus,
\begin{equation}
\label{eq:f arith rand wav def}
f(x) =f_{n}(x)= \sum\limits_{\lambda\in \Ec_{n}}c_{\lambda}e(\langle \lambda, x\rangle),
\end{equation}
where $e(z)=e^{2\pi i z}$ and
\begin{equation}
\label{eq:Ecn lattice circle def}
\Ec_{n}=\{\lambda=(\lambda_{1},\lambda_{2})\in\Z^{2}: \:\|\lambda\|^{2}=n\}
\end{equation}
is the set of all representations of the integer $n=\lambda_1^2+\lambda_2^2$
as a sum of two integer squares, and $c_\lambda$ are standard Gaussian random variables\footnote{After understanding the Gaussian case, one may try non-Gaussian ensembles, see e.g. \cite{CNNV}.}, identically distributed and independent save for the constraint
\begin{equation}
\label{eq:c(-lam),c(lam)}
c_{-\lambda} = \overline{c_{\lambda}}\;,
\end{equation}
making $f_n$ real valued eigenfunctions of the Laplacian with eigenvalue
\begin{equation}
\label{eq:E=4pi^2n}
E=4\pi^2n
\end{equation}
for every choice of the coefficients $\{c_{\lambda}\}_{\lambda\in\Ec_{\lambda}}$
(i.e. for every sample point).

Equivalently $f_{n}:\Tb^{2}\rightarrow\R$
is a centred Gaussian random field with covariance
\begin{equation}
\label{eq:r covar def}
r(x,y)=r_{n}(y-x) = \frac{1}{N_{n}}\sum\limits_{\lambda\in\Ec_{n}}e(\langle \lambda,y-x \rangle).
\end{equation}
Since $r(x,y)$ depends only on $y-x$, the random field $f_{n}$ is {\em stationary}, meaning that
for every translation $$\tau_{z}:f_{n}(\cdot)\mapsto f_{n}(\cdot+z)$$ with $z\in\Tb^{2}$, the law
of $\tau_{z} f_{n}$ equals the law of $f_{n}$:
\begin{equation}
\label{eq:stat law}
\tau_{z} f_{n} \stackrel{d}{=} f_{n}.
\end{equation}
This, in turn, is equivalent to the law
of the Gaussian multivariate vector $(f_{n}(x_{1}),\ldots, f_{n}(x_{k}))$ being equal to the
law of the vector $(f_{n}(x_{1}+z),\ldots, f_{n}(x_{k}+z))$ for every $x_{1},\ldots, x_{k}\in\Tb^{2}$,
$z\in\Tb^{2}$.

In \cite{RW2008} we studied the statistics of the length of the nodal line of $f_{n}$.
Since then, very refined data has been obtained on the nodal structure of such random eigenfunctions
(see e.g. \cite{K-K-W, MPRW, PR, Rozenshein, KW}).

We will compute the expected value of $N_\zeta$ for arithmetic random waves. The answer depends on the distribution of lattice points on the circle of radius $\sqrt{n}$.
Let $\mu_{n}$ be the atomic measure on the unit circle given by
\begin{equation*}
\mu_{n}=\frac{1}{r_{2}(n)}\sum\limits_{\lambda\in\Ec_{n}} \delta_{\lambda/\sqrt{n}},
\end{equation*}
where $r_{2}(n):=\#\Ec_{n}$, and let
\[
\widehat{\mu_{n}}(k) 
=\frac{1}{r_{2}(n)}\sum\limits_{\lambda=(\lambda_1,\lambda_2)\in\Ec_{n}} \left(\frac{\lambda_1+i\lambda_2}{\sqrt{n}}\right)^k\in\R
\]
be its Fourier coefficients.
\begin{theorem}
\label{thm:E[Nxi] d=2}
For $\zeta=e^{i\theta} \in S^1$, the expected value of $N_\zeta(f)$ for the arithmetic random wave \eqref{eq:f arith rand wav def} is
\begin{equation}
\label{eq:E[Nxi] d=2}
\E[N_{\zeta}] = \frac{1}{\sqrt{2}}n\cdot \left(1+\widehat{\mu_{n}}(4)\cdot \cos(4\theta)\right)^{1/2}.
\end{equation}
\end{theorem}

The statement \eqref{eq:E[Nxi] d=2} of Theorem \ref{thm:E[Nxi] d=2} is valid even if the r.h.s. of \eqref{eq:E[Nxi] d=2} vanishes, i.e.
if $$\widehat{\mu_{n}}(4)\cdot \cos(4\theta)=-1:$$ either
$$\mu_{n} = \frac{1}{4}\left(\delta_{\pm 1}+\delta_{\pm i}\right)$$
(``Cilleruelo measure") and $\theta = \pm\frac{\pi}{4},\pm \frac{3\pi}{4}$, or $\mu_{n}$ is the rotation by $\frac{\pi}{4}$ of the latter measure
(``tilted Cilleruelo") and $\zeta$ is parallel to one of the axes.
These cases are exceptional in the following sense: It is known ~\cite{NS,KW} that for every probability measure $\mu$ on the unit circle $\Sc^{1}$ there exists a constant $c_{NS}(\mu)\ge 0$ (the ``Nazarov-Sodin constant")
such that if the measures $\mu_n$ converge weak-$*$ to $\mu$, 
then the expectation of the number $\Cc(f_{n})$ of nodal domains of $f_n$ is
\begin{equation*}
\E[\Cc(f_{n})] = (c_{NS}(\mu)+o(1))\cdot n.
\end{equation*}
Moreover, the Nazarov-Sodin constant $c_{NS}(\mu)=0$ vanishes, if and only if $\mu$ is  one of these exceptional measures
\cite{KW}. In that case it was shown \cite{KW} that
most of the nodal components are long and mainly parallel to one of the axes (perhaps, after rotation by $\frac{\pi}{4}$);
with accordance to the above, our computation \eqref{eq:E[Nxi] d=2} implies in particular that $c_{NS}(\mu)=0$ for
$\mu$ (tilted) Cilleruelo measure, i.e. the ``if" part of the aforementioned statement from ~\cite{KW}.


\bigskip

One can study an analogous quantity $N_\zeta(f)$ for eigenfunctions on the $d$-dimensional torus $\Tb^d = \R^d/\Z^d$,
$d\ge 3$ with eigenvalue $4\pi^2n$. We can establish a result analogous to Theorem \ref{thm:E[Nxi] d=2} in the higher dimensional
case, showing that for $d\geq 3$,
\begin{equation*}
\E[N_{\zeta}] \sim  C_{d} n^{d/2}  ,\quad n\rightarrow\infty,
\end{equation*}
for some positive constant $C_{d}>0$ independent of $\zeta$, assuming that $n\neq 0,4,7 \bmod 8$ if $d=3$, and $n\neq 0\bmod 8$ if $d=4$.

\subsection*{Acknowledgements}
We thank Jerry Buckley, Suresh Eswarathasan, Manjunath Krishnapur, Mark Shusterman and Mikhail Sodin for their comments.
The work was supported  by the European Research Council under the European Union's Seventh
Framework Programme (FP7/2007-2013)/ERC grant agreement n$^{\text{o}}$~320755 (Z.R.) and  n$^{\text{o}}$~335141 (I.W.).

\section{Deterministic upper bound: proof of Theorem \ref{thm:upper bnd class}}

\label{sec:det bnd proof}

Before giving a proof for Theorem \ref{thm:upper bnd class} we will need some preparatory results, all related to
the identification of the trigonometric polynomials on $\Tb^{2}$ with Laurent polynomials in $\C[z_{1},z_{2}]$,
via the natural embedding $\Tb^{2}=\Sc^{1}\times \Sc^{1} \hookrightarrow\C^{2}$ (see \eqref{eq:Tb->C2 embed} below).

\subsection{From trigonometric polynomials to (Laurent) polynomials}

\begin{definition}

\begin{enumerate}

\item Let $\Pc$ be the space of all complex valued trigonometric polynomials on $\Tb^{2}$.
We define an operator $\Phi:\Pc\rightarrow\C[z_{1},z_{2},z_{1}^{-1},z_{2}^{-1}]$ between $\Pc$ and the complex Laurent polynomials
in the following way. For $g:\Tb^{2}\rightarrow\R$ a trigonometric polynomial
\begin{equation}
\label{eq:g=trig pol}
g(x)=\sum\limits_{\substack{\lambda\in\Z^{2}\\\text{finite sum}} }c_{\lambda}e^{2\pi i \langle \lambda, x \rangle},
\end{equation}
we associate the Laurent polynomial
$\widetilde{G}=\Phi(g)\in \C[z_{1},z_{1}^{-1},z_{2},z_{2}^{-1}]$ via the embedding $\Tb^{2}=\Sc^{1}\times \Sc^{1} \hookrightarrow\C^{2}$
\begin{equation}
\label{eq:Tb->C2 embed}
(x_{1},x_{2})\mapsto (z_{1},z_{2})=\left(e^{2\pi i x_{1}}, e^{2\pi i x_{2}}\right),
\end{equation}
or, explicitly,
\begin{equation*}
\widetilde{G}(z) = g(x) = \sum\limits_{\lambda\in\Z^{2}}c_{\lambda}z^{\lambda},
\end{equation*}
where for $z=(z_{1},z_{2})\in\C^{2}$ and $\lambda\in\Ec_{n}$ we denote $z^{\lambda}:=z_{1}^{\lambda_{1}}\cdot z_{2}^{\lambda_{2}}$.

\item For $k=1,2$ let $D_{k}:\C[z_{1},z_{2}] \rightarrow\C[z_{1},z_{2}]$ be the operator
$$D_{k}:p(z)\mapsto z_{k}\frac{\partial p(z)}{\partial z_k} .
$$

\item For $\xi\in\Sc^{1}$ denote the operator
$$D_{\xi} = \langle (D_{1},D_{2}), \xi\rangle = \xi_{1}D_{1}+\xi_{2}D_{2}.$$

\end{enumerate}

\end{definition}

The following properties are immediate from the definitions:

\begin{lemma}
\label{lem:dir der corr pol}

\begin{enumerate}

\item

For every $\xi\in\Sc^{1}$ the operator $D_{\xi}$ (in particular, $D_{1}$ and $D_{2}$) is a derivation, i.e. it is a linear operator
satisfying the Leibnitz law
$$D_{\xi}(p(z)q(z)) = D_{\xi}p(z)\cdot q(z)+p(z)\cdot D_{\xi}q(x).$$

\item For every $g$, a trigonometric polynomial as in \eqref{eq:g=trig pol}, and $x=(x_{1},x_{2})\in \Tb^{2}$, we have
\begin{equation}
\label{eq:g(x)=G(z)}
g(x) = (\Phi g)(z) = \widetilde{G}(z),
\end{equation}
where $z=z(x)$ is given by \eqref{eq:Tb->C2 embed} and $\widetilde{G}=\Phi g$.

\item
For $\xi\in\Sc^{1}$, if $\widetilde{G}=\Phi g$, then
\begin{equation}
\label{eq:der corr derivation}
\frac{1}{2\pi i}\Phi(\partial_{\xi}g) = D_{\xi} \widetilde{G},
\end{equation}
i.e. if under $\Phi$, $g$ maps to $g\mapsto \widetilde{G}$, then its (normalised) directional derivative
$\frac{1}{2\pi i}\partial_{\xi}g$ maps to $D_{\xi}\widetilde{G}$.

\end{enumerate}

\end{lemma}

\subsection{Auxiliary lemmas}

\begin{lemma}
\label{lem:nonsing no mult}
Let $g:\Tb^{2}\rightarrow\R$ be
a trigonometric polynomial
\eqref{eq:g=trig pol}, $x_{0}\in g^{-1}(0)$ a nonsingular zero,
$\widetilde{G}=\Phi(g)\in\C[z_{1},z_{1}^{-1},z_{2},z_{2}^{-1}]$,
and $G(z)=z^{\delta}\widetilde{G}(z) \in \C[z_{1},z_{2}]$, so that
$$G(z_{0}) = g(x_{0}) = 0,$$ where $z_{0}=z(x_{0})\in\C^{2}$ is the point corresponding to $x_{0}$
via \eqref{eq:Tb->C2 embed}. Suppose also that $P\mid G$ is an irreducible factor of $G$ such that $P(z_{0}) = 0$.
Then $P^{2}\nmid G$.
\end{lemma}

\begin{proof}
Assume by contradiction that, under the assumptions of Lemma \ref{lem:nonsing no mult}, we have that
\begin{equation}
\label{eq:P^2|G}
P^{2}\mid G
\end{equation}
we then claim that
in this case necessarily $\nabla g(x_{0}) = 0$, contradicting the non-singularity of $x_{0}$ as a zero
of $g$. We show that $\frac{\partial g}{\partial x_{k}}(x_{0}) = 0$,  $k=1,2$.

Since $D_{k}$ is a derivation in $\C[z_{1},z_{2}]$, and by \eqref{eq:der corr derivation} we have that
\begin{equation}
\label{eq:Phi(der) chain}
\Phi\left(\frac{1}{2\pi i}\frac{\partial g}{\partial_{x_{k}}}\right) = D_{k}\widetilde{G} = D_{k}(z^{-\delta}G) = z^{-\delta}\cdot D_{k}G +
G\cdot D_{k}z^{-\delta}.
\end{equation}
Since both $G$ and $D_{k}G$ are divisible by $P$ by our assumption \eqref{eq:P^2|G}, we have
$G(z_{0})=D_{k}G(z_{0}) = 0$. Substituting this into \eqref{eq:Phi(der) chain}, and bearing in mind \eqref{eq:g(x)=G(z)}, this yields that
$\frac{\partial g}{\partial x_{k}}(x_{0}) = 0$.
Thus $x_{0}$ is a singular zero of $g$, contradicting our assumption.
\end{proof}

\begin{lemma}
\label{lem:irred fact scal der}
Let $\widetilde{G}\in\C[z_{1},z_{1}^{-1},z_{2},z_{2}^{-1}]$ be a Laurent polynomial,
$\delta\in\Z^{2}_{\geq 0}$ so that
\begin{equation}
\label{eq:G=zdelt*tildG}
G(z)=z^{\delta}\widetilde{G}(z) \in \C[z_{1},z_{2}]
\end{equation}
is a polynomial, with $\delta$ minimal in  the sense that $z_j\nmid G$. Let
$\widetilde{Q}_{\xi}(z) = D_{\xi}(G)(z)$
and
\begin{equation}
\label{eq:Q=zdelt*tildQ}
Q_{\xi}(z):=z^{\delta}\cdot\widetilde{Q}_{\xi}(z)\in\C[z_{1},z_{2}].
\end{equation}
Suppose that
\begin{equation}
\label{eq:P|gcd(G,Q)}
P\mid \gcd(G,Q_{\xi})
\end{equation}
is an irreducible polynomial, such that $P^{2}\nmid G$. Then necessarily
$D_{\xi}P$ is a scalar multiple of $P$, i.e. there exists $t\in\C$ so that
\begin{equation}
\label{eq:phiP=tP}
D_{\xi}P = t\cdot P.
\end{equation}
\end{lemma}

\begin{proof}

First, since by Lemma \ref{lem:dir der corr pol}, $D_{\xi}$ is a derivation, we have that
\begin{equation}
\label{eq:phixiG chain}
\begin{split}
D_{\xi}G &= D_{\xi}(z^{\delta}\cdot z^{-\delta}G) = D_{\xi}(x^{\delta}\cdot \widetilde{G}) =
D_{\xi}(x^{\delta})\cdot \widetilde{G} + x^{\delta}\cdot D_{\xi}(\widetilde{G}) \\&= \langle\delta,\xi\rangle x^{\delta}\cdot\widetilde{G}
+x^{\delta}\widetilde{Q}_{\xi} = \langle\delta,\xi\rangle G + Q_{\xi},
\end{split}
\end{equation}
by \eqref{eq:G=zdelt*tildG} and \eqref{eq:Q=zdelt*tildQ}.
Hence, since, by assumption \eqref{eq:P|gcd(G,Q)}, both summands on the r.h.s. of \eqref{eq:phixiG chain} are divisible by $P$, so is $D_{\xi}G$, i.e.
\begin{equation}
\label{eq:P|phixi(G)}
P\mid D_{\xi}G.
\end{equation}

Now let us write
\begin{equation}
\label{eq:G=PA}
G=P\cdot A
\end{equation}
for some $A\in\C[z_{1},z_{2}]$; since by assumption $P$ is irreducible, and $P^{2}\nmid G$ by
Lemma \ref{lem:nonsing no mult}, this necessarily implies
\begin{equation}
\label{eq:gcd(P,A)=1}
\gcd(P,A)=1.
\end{equation}
Applying the derivation $D_{\xi}$ on \eqref{eq:G=PA}
we obtain:
\begin{equation*}
D_{\xi}G = D_{\xi}(P)\cdot A+P\cdot D_{\xi}A,
\end{equation*}
which, together with \eqref{eq:G=PA} yields that
\begin{equation*}
P\mid D_{\xi}(P)\cdot A,
\end{equation*}
which, by \eqref{eq:gcd(P,A)=1}, forces
\begin{equation}
\label{eq:P|dir der P}
P\mid D_{\xi}(P).
\end{equation}
Note that if
\begin{equation*}
P(z)=\sum\limits_{\alpha\in\Z_{\ge 0}^{2}}p_{\alpha}z^{\alpha}
\end{equation*}
is a finite sum, then
\begin{equation*}
D_{\xi}(P)(z)=\sum\limits_{\alpha\in\Z_{\ge 0}^{2}}\langle\xi,\alpha \rangle p_{\alpha}z^{\alpha}
\end{equation*}
is of degree at most  the degree of $P$. Hence \eqref{eq:P|dir der P} implies that
$D_{\xi}P$ is a scalar multiple of $P$.
\end{proof}

\begin{lemma}
\label{lem:P=phixiP xi rat,P}
Let $\xi\in\Sc^{1}$, $t\in\C$, and $P\in\C[z_{1},z_{2}]$ nonconstant irreducible polynomial such that
$z_{1},z_{2}\nmid P$, and
\begin{equation}
\label{eq:phixiP=tP}
D_{\xi}P = t\cdot P.
\end{equation}
Then the following hold:
\begin{enumerate}

\item The direction $\xi$ is rational (i.e. the vector $\xi$ is a multiple of a rational vector).

\item The polynomial $P$ is necessarily of the form
\begin{equation}
\label{eq:P=p1z1^k1+p2z2^k2}
P(z) = p_{1}z_{1}^{k_{1}}+p_{2}z_{2}^{k_{2}}
\end{equation}
for some $p_{1},p_{2}\in\C\backslash \{0\}$, and $(k_{2},k_{1})\in\Z_{\ge 0}^{2}$ is a primitive vector (unique up to sign) satisfying
$$\frac{(k_{2},k_{1})}{\|(k_{2},k_{1})\|} = \pm \xi.$$

\end{enumerate}

\end{lemma}

\begin{proof}

Writing $P$ as a finite sum
\begin{equation*}
P(z) = \sum\limits_{\alpha}\langle\xi,\alpha \rangle p_{\alpha}z^{\alpha},
\end{equation*}
(the finite sum over $\alpha\in\Z_{\ge 0}^{2}$), the equality \eqref{eq:phixiP=tP} is equivalent to
\begin{equation*}
\langle \xi,\alpha\rangle \cdot p_{\alpha}=t\cdot p_{\alpha}
\end{equation*}
for every $\alpha\in\Z_{\ge 0}^{2}$, i.e.
\begin{equation}
\label{eq:<xi,alpha>=t}
\langle \xi,\alpha\rangle = t
\end{equation}
for every $\alpha\in\Z_{\ge 0}^{2}$ with $p_{\alpha} \ne 0$. Note that $P$ is not a monomial
(as otherwise $P$ would be divisible by either $z_{1}$ or $z_{2}$), hence \eqref{eq:<xi,alpha>=t}
is valid for at least two distinct $\alpha$. Therefore, for these $\alpha$, one has
\begin{equation*}
\langle \xi,\alpha-\alpha'\rangle=0,
\end{equation*}
which forces $\xi$ to be {\em rational}, i.e. yields the first statement of Lemma \ref{lem:P=phixiP xi rat,P}.

\vspace{2mm}

Now assume that the rational vector $\xi=\frac{u}{\|u\|}$ is a multiple of a primitive integer vector $u\in\Z^{2}$. We may
then rewrite \eqref{eq:<xi,alpha>=t} as
\begin{equation}
\label{eq:<u,alpha>=s}
\langle u,\alpha\rangle = s,
\end{equation}
with $s=\|u\|\cdot t$, uniquely determined by $\xi$ and $t$, and to have any solution to \eqref{eq:<u,alpha>=s}, necessarily $s\in\Z$.
The integer solutions to \eqref{eq:<u,alpha>=s}, considered as an equation in $\alpha$, are
\begin{equation}
\label{eq:alph=alph0+kv}
\alpha=\alpha^{0}+k\cdot v,
\end{equation}
where $\alpha^{0}$ is a particular solution to \eqref{eq:<u,alpha>=s},
and $v\in\Z^{2}$ is the primitive integer vector orthogonal to $u$, unique up to sign, some of whose coordinates might be negative.
Note that \begin{equation*}
\zeta =  \frac{v}{\|v\|}
\end{equation*}
is a unit vector orthogonal to $\xi$.

Since the collection $$\{\alpha\in\Z^{2}:\: p_{\alpha}\ne 0\}$$ is finite (corresponding to a finite collection of $k$ in
\eqref{eq:alph=alph0+kv}), we can choose $\alpha^{0}$ a particular solution of
\eqref{eq:<u,alpha>=s} so that
\begin{equation}
\label{eq:palpha0 ne 0}
p_{\alpha^{0}}\ne 0,
\end{equation}
and the numbers $k$ in \eqref{eq:alph=alph0+kv} satisfy $0\le k\le K$ for some
$K>0$; by \eqref{eq:palpha0 ne 0} we necessarily have $\alpha_{0}\in\Z_{\ge 0 }^{2}$. We may then write:
\begin{equation}
\label{eq:P(z)=xalpha0*Q(zv)}
P(z) = \sum\limits_{k=0}^{K}p_{k}z^{\alpha^{0}+k\cdot v} = z^{\alpha^{0}}\cdot \sum\limits_{k=0}^{K}(z^{v})^{k} =
z^{\alpha^{0}}\cdot Q(z^{v}),
\end{equation}
where $Q(w)\in\C[w]$ is a (one variable) complex polynomial, which, by above, is not a monomial.

We claim that the irreducibility of $P$ implies the irreducibility of $Q$, which, in turn, implies that $Q$ is {\em linear}.
For if $Q$ were reducible, we could write
\begin{equation}
\label{eq:Q=AB}
Q(w)=A(w)\cdot B(w)
\end{equation}
for some nonconstant polynomials $A,B\in\C[w]$. Substituting \eqref{eq:Q=AB} into \eqref{eq:P(z)=xalpha0*Q(zv)}, we obtain
\begin{equation}
\label{eq:P=zalphaAB}
P(z) = z^{\alpha^{0}} \cdot A(z^{v})B(z^{v}).
\end{equation}
As one or both components of $v$ might be negative, \eqref{eq:P=zalphaAB}
does not immediately imply that $P$ is reducible. Write
\begin{equation*}
A(z^{v}) = z^{-\alpha^{1}}\widetilde{A}(z),
\end{equation*}
\begin{equation*}
B(z^{v}) = z^{-\alpha^{2}}\widetilde{B}(z),
\end{equation*}
where $\alpha^{1},\alpha^{2}\in\Z_{\ge 0}^{2}$ are minimal so that $\widetilde{A}(z), \widetilde{B}(z)\in\C[z]$ are polynomial,
so that $\widetilde{A},\widetilde{B}$ are not divisible by $z_{1},z_{2}$.
We then have
\begin{equation}
\label{eq:P=za0-a1-a2AB}
P(z) = z^{\alpha^{0}-\alpha^{1}-\alpha^{2}} \cdot \widetilde{A}(z)\widetilde{B}(z).
\end{equation}
Since $P$ is not divisible by $z_{1},z_{2}$ and neither are $A$ and $B$, the equality \eqref{eq:P=za0-a1-a2AB} implies that
$\alpha^{0}-\alpha^{1}-\alpha^{2} = 0$, so that
\begin{equation*}
P(z)=\widetilde{A}(z)\cdot\widetilde{B}(z)
\end{equation*}
is a factorization of $P$ into nonconstant polynomials, contradicting the assumption that $P$ is irreducible, and hence
$Q$ as in \eqref{eq:P(z)=xalpha0*Q(zv)} is itself irreducible in $\C[w]$, so
\begin{equation}
\label{eq:Q=q0+q1w}
Q(w)=q_{0}+q_{1}w
\end{equation}
with $q_{0},q_{1}\in \C^{*}$, is linear.

Substituting  \eqref{eq:Q=q0+q1w} into \eqref{eq:P(z)=xalpha0*Q(zv)} gives
\begin{equation}
\label{eq:P=q0za0+q1za0+v}
P(z) = q_{0}z^{\alpha_{0}} + q_{1}z^{\alpha_{0}+v},
\end{equation}
and $\alpha_{0},\alpha_{0}+v \in \Z_{\ge 0 }^{2}$. Since $\alpha_{0}\ne \alpha_{0}+v$ and $z_{1},z_{2}\nmid P$,
the form \eqref{eq:P=q0za0+q1za0+v} of $P$ reduces to \eqref{eq:P=p1z1^k1+p2z2^k2}, and it also forces
\begin{equation*}
v = (-k_{1},k_{2}),
\end{equation*}
hence $(k_{2},k_{1})$ is a primitive lattice point of $\Z^{2}$, co-linear with $\xi$.
\end{proof}

\subsection{Proof of Theorem \ref{thm:upper bnd class}}

\begin{proof}

Let $f=f_{n}$ be 
a
toral eigenfunction \eqref{eq:f arith rand wav def} (it is a monochromatic trigonometric polynomial whose frequency set
$\Ec_{n}$ is given by \eqref{eq:Ecn lattice circle def}), and
$$\widetilde{G}=\Phi(f)\in\C[z_{1},z_{1}^{-1},z_{2},z_{2}^{-1}]$$ be the Laurent polynomial associated to $f$ as in
Lemma \ref{lem:dir der corr pol}, so that
\begin{equation}
\label{eq:G(z)=g(x0)}
\widetilde{G}(z) = f(x) = \sum\limits_{\lambda\in\Ec_{n}}c_{\lambda}z^{\lambda} .
\end{equation}
Note that for $\lambda\in \Ec_{n}$, we have $|\lambda_1|+\lambda_2|\leq \sqrt{2n}$.
To make $\widetilde{G}$ into a polynomial in $\C[z_{1},z_{2}]$ we multiply $\widetilde{G}$ by a monomial $z^{\delta}$ with $\delta\in\Z_{\ge 0}^{2}$ satisfying
\begin{equation}
\label{eq:|delta|<=sqrt(E)}
\delta_{1}+\delta_{2} \le \sqrt{2n},
\end{equation}
to write
\begin{equation}
\label{eq:F=zdelta*tilde(F)}
G(z)=z^{\delta}\widetilde{G}(z),
\end{equation}
with $\delta$ minimal, so that, in particular, $G(z)$ is {\em not divisible} by $z_{1}$  or $z_{2}$.
By \eqref{eq:f arith rand wav def} and \eqref{eq:|delta|<=sqrt(E)},
we have
\begin{equation}
\label{eq:deg(G)<=2sqrt(E)}
\deg (G) \le 2\sqrt{2}\cdot\sqrt{n}.
\end{equation}
Now let $\widetilde{Q}_{\xi}=\frac{1}{2\pi i}\Phi(\partial_{\xi}f)$ be the Laurent polynomial corresponding to the directional
derivative $\partial_{\xi}f(x)$ of $f$ where $\xi=\zeta^\perp$ is orthogonal to $\zeta$. By Lemma \ref{lem:dir der corr pol} we have
\begin{equation}
\label{eq:Qtild=der G}
\widetilde{Q}_{\xi}(z) = D_{\xi}(\widetilde{G}(z)) = \frac{1}{2\pi i}\Phi(\partial_{\xi}f)(z)=
\sum\limits_{\lambda\in\Ec_{n}}\langle \lambda,\xi\rangle c_{\lambda}z^{\lambda},
\end{equation}
and
\begin{equation}
\label{eq:Q(z)=zdel Qtild(z)}
Q_{\xi}(z):=z^{\delta}\cdot\widetilde{Q}_{\xi}(z)\in\C[z_{1},z_{2}]
\end{equation}
with $\delta$ same as in \eqref{eq:F=zdelta*tilde(F)}, is a polynomial of degree
\begin{equation}
\label{eq:deg(Qxi)<=2sqrt(E)}
\deg(Q_{\xi})\le 2\sqrt{2}\cdot\sqrt{n},
\end{equation}
though might be divisible by $z_{1}$ or $z_{2}$.
By \eqref{eq:G(z)=g(x0)}, \eqref{eq:F=zdelta*tilde(F)}, \eqref{eq:Qtild=der G},
and \eqref{eq:Q(z)=zdel Qtild(z)}, for some $x_{0}\in\Tb^{2}$ we have
$$f(x_{0})=\partial_{\xi}f(x_{0})=0,$$ (without imposing $\nabla f(x_{0})\ne 0$), if and only if $z_{0}=z(x_{0})$ is a joint zero of both $G$ and $Q_{\xi}$, i.e. $$G(z_{0})=Q_{\xi}(z_{0}) = 0.$$

Now let
\begin{equation*}
D= \gcd(G,Q_{\xi})
\end{equation*}
be the greatest common divisor of $G$ and $Q_{\xi}$
\begin{equation*}
G(z) = A(z)\cdot D(z)
\end{equation*}
and
\begin{equation*}
Q_{\xi}(z) = B(z)\cdot D(z),
\end{equation*}
where
\begin{equation}
\label{eq:gcd(A,B)=1}
\gcd(A,B) = 1
\end{equation}
and
\begin{equation}
\label{eq:deg(A),deg(A),deg(B)<<sqrt(E)}
\deg(D),\,\deg(A)\le \deg(G)\le 2\sqrt{2}\cdot\sqrt{n}, \;\;\; \deg(B)\le \deg(Q_{\xi}) \le 2\sqrt{2}\cdot\sqrt{n}
\end{equation}
by \eqref{eq:deg(G)<=2sqrt(E)} and \eqref{eq:deg(Qxi)<=2sqrt(E)}, and,
by the above, we are interested in $z=(z_{1},z_{2})\in \C^{2}$, so that $|z_{1}|=|z_{2}|=1$
and $G(z)=Q_{\xi}(z) = 0$.

Given $z_{0}\in\C^{2}$ we have that $G(z_{0})= Q_{\xi}(z_{0}) = 0$,
if and only if either
\begin{equation*}
A(z_{0})=B(z_{0}) = 0,
\end{equation*}
or $D(z_{0})=0$ (both cannot occur simultaneously).
Denote
\begin{equation}
\label{eq:Z1 coprm def}
\Zc^{1}(G,Q_{\xi}) := \{ z\in\C^{2}:\: A(z) = B(z) = 0\}
\end{equation}
and
\begin{equation}
\label{eq:Z2 gcd def}
\Zc^{2}(G,Q_{\xi}) := \{ z\in\C^{2}:\: D(z) = 0\},
\end{equation}
the nodal directional points of the {\em first and second type} respectively.
The meaning of the above is that, under the embedding \eqref{eq:Tb->C2 embed}
of $\Sc^{1}\times \Sc^{1} \subseteq \C^{2} $,
\begin{equation}
\label{eq:dir nod pnt decomp pol}
\{x\in\Tb^{2}; f(x)= \langle \nabla f, \xi \rangle = 0\} \mapsto (\Zc^{1}(G,Q_{\xi}) \cup \Zc^{2}(G,Q_{\xi}))\cap \Sc^{1}\times\Sc^{1}.
\end{equation}
Hence understanding of $$\Zc^{1}(G,Q_{\xi}) \cup \Zc^{2}(G,Q_{\xi})$$ will also allow for bounding the size of
the l.h.s. of \eqref{eq:dir nod pnt decomp pol}; note that, unlike the definition \eqref{eq:Nzeta=f=nablf=zeta def}
of $N_{\zeta}$, the l.h.s. of \eqref{eq:dir nod pnt decomp pol}
includes singular points of $f^{-1}(0)$, having no bearing on giving an upper bound for $N_{\zeta}$
via one for the r.h.s. of \eqref{eq:Nzeta=f=nablf=zeta def}.
Since $A$ and $B$ are co-prime by \eqref{eq:gcd(A,B)=1}, and bearing in mind \eqref{eq:deg(A),deg(A),deg(B)<<sqrt(E)} and the definition
\eqref{eq:Z1 coprm def}, it follows that $\Zc^{1}(G,Q_{\xi})$
consists of finitely many isolated points, and its cardinality is bounded, by B\'ezout's Theorem\footnote{which states that if $A,B\in \C[z_1,z_2]$ are co-prime polynomials, then the number of common zeros of $A$ and $B$ is bounded by $\deg A\cdot \deg B$.}
\begin{equation}
\label{eq:|Z1|<=4E}
|\Zc^{1}(G,Q_{\xi})| \le \deg (A)\cdot \deg(B) \le 8n = \frac{2E}{\pi^{2}},
\end{equation}
on using \eqref{eq:deg(A),deg(A),deg(B)<<sqrt(E)} and \eqref{eq:E=4pi^2n}.

Now we turn to understanding $\Zc^{2}(G,Q_{\xi})$ as in \eqref{eq:Z2 gcd def}. Let $P|D$ be an irreducible divisor of $D=\gcd(G,Q_{\xi})$,
and let $x_{0}\in\Ac_{\zeta}(f)\in\Tb^{2}$ be a {\em nonsingular} nodal directional point so that $P(z_{0}) = 0$, where $z_{0}=z(x_{0})$,
the map in \eqref{eq:Tb->C2 embed}. Then, thanks to Lemma \ref{lem:nonsing no mult}, $P^{2}\nmid G$, so that we may apply Lemma
\ref{lem:irred fact scal der} to deduce that
\begin{equation}
\label{eq:phixP=tP}
D_{\xi}P = t\cdot P,
\end{equation}
for some scalar $t\in\C$. By invoking Lemma \ref{lem:P=phixiP xi rat,P}, the equality \eqref{eq:phixP=tP}
in turn implies that $\xi$ is a rational direction, and
\begin{equation*}
P(z) = p_{1}z_{1}^{k_{1}}+p_{2}z_{2}^{k_{2}},
\end{equation*}
where the {\em primitive} vector
\begin{equation}
\label{eq:k2,k1 co-lin xi}
(k_{2},k_{1})\in\Z^{2}_{\ge 0} \text{ is co-linear to } \xi = \zeta^\perp,\quad {\rm i.e. \;orthogonal\; to} \;\zeta.
\end{equation}

Thus
\begin{equation}
\label{eq:D=prod(Pi) E}
D = \left(\prod\limits_{i=1}^{K} P_{j}(z)\right)\cdot E(z),
\end{equation}
where for every $j=1,2,\ldots K$ the polynomial $P_{j}$ is of the form
\begin{equation*}
P_{j}(z) = p_{1;j}z_{1}^{k_{1}}+p_{2;j}z_{2}^{k_{2}},
\end{equation*}
for some $p_{1;j},p_{2;j}\in\C$, and
$E(z)$ is the product of irreducible factors $P\mid D$ of $D$ so that $P^2\mid D$ (corresponding to the singular points $x_{0}\in\Tb^{2}$),
and those irreducible $P\mid D$ that don't vanish on $\Sc^{1}\times \Sc^{1}\subseteq \C^{2}$. It then follows that
\begin{equation*}
K \le \frac{\deg(D)}{\max(k_{1},k_{2})} \le 2\frac{\sqrt{n}}{h(\xi)}=2\frac{\sqrt{n}}{h(\zeta)}
\end{equation*}
by \eqref{eq:deg(A),deg(A),deg(B)<<sqrt(E)}.

Now using \eqref{eq:g(x)=G(z)} on
\eqref{eq:D=prod(Pi) E}, \eqref{eq:Z2 gcd def},
we have that, under the embedding \eqref{eq:dir nod pnt decomp pol}, the zeros of $D(z)$ correspond to the zeros of
\begin{equation}
\label{eq:d(x)=map D}
d(x):=D(z(x)) = \left(\prod \limits_{j=1}^{K} \left(p_{1;j}e^{2\pi i k_{1} x_1} +p_{2;j}e^{2\pi i k_{2}x_2}\right)\right) \cdot
\tilde E(x),
\end{equation}
where $\tilde E(x):\Tb^{2}\rightarrow\R$ is the trigonometric polynomial corresponding to $E(z)=\tilde E(x)$, that only has singular zeros.
Let
\begin{equation*}
h_{j}(x):= p_{1;j}e^{2\pi i x_{1}k_{1}} +p_{2;j}e^{2\pi i x_{2}k_{2}}
\end{equation*}
be a   factor of \eqref{eq:d(x)=map D}; by construction \eqref{eq:D=prod(Pi) E}
we know a priori that the zero locus of $h_{j}$ on $\Tb^{2}$ is non-empty. In this case, necessarily
$|p_{1;j}|=|p_{2;j}|$, and upon writing $$-\frac{p_{2;j}}{p_{1;j}} =: e^{2\pi i \varphi}$$ for $\varphi\in [0,1)$, the zero locus
of $h_{j}$ is given by
\begin{equation*}
\begin{split}
h_{j}^{-1}(0) &= \left\{(x_{1},x_{2})\in\Tb^{2}: \: e(x_{1}k_{1}-x_{2}k_{2}) = -\frac{p_{2;j}}{p_{1;j}} \right\} \\&=
\left\{(x_{1},x_{2})\in\Tb^{2}: \: x_{1}k_{1}-x_{2}k_{2} = \varphi \mod{1} \right\},
\end{split}
\end{equation*}
hence is a closed geodesic in $\Tb^{2}$ (it has a single connected component, since, by assumption,
$\gcd(k_{1},k_{2})=1$), orthogonal to $(k_{1},-k_{2})$, of length $\sqrt{k_1^2+k_2^2}$, and, recalling \eqref{eq:k2,k1 co-lin xi},
the geodesic $h_{j}^{-1}(0)$ is orthogonal to $\zeta$.
In summary, under the embedding \eqref{eq:Tb->C2 embed}, the nonsingular points on $f^{-1}(0)$ corresponding to
the set $\Zc^{2}(G,Q_{\xi})\cap (\Sc^{1}\times \Sc^{1})$ consist of $$\le 2\frac{\sqrt{n}}{h(\xi)}=\frac{\sqrt{E}}{h(\zeta)}$$
closed geodesics orthogonal to $\zeta$, concluding the statement of Theorem \ref{thm:upper bnd class}.
\end{proof}

\section{Expected nodal direction number for arithmetic random waves: proof of Theorem \ref{thm:E[Nxi] d=2}}

In this section, we compute the expected value of $N_\zeta$ for arithmetic random waves. The formal computation is along the lines of Swerling's paper \cite{Swerling}, but his argument relied on several assumptions, some implicit, on the nature of the relevant Gaussian field, which are difficult to isolate and check separately. Thus we carry out the computation {\it ab initio}.

\subsection{Proof of Theorem \ref{thm:E[Nxi] d=2}}
\label{sec:mean arw d=2}

\begin{proof}

Let $\xi = \zeta^{\perp}$ be the orthogonal vector to $\zeta$, and define,
\begin{equation*}
\widetilde{N_{\zeta}}(f) = \#\left\{x\in\Tb^{2}:\: f(x)=\langle\nabla f(x),\xi\rangle = 0  \right\}
\end{equation*}
to be the size of the set $\Ac_{\zeta}(f)$ in \eqref{eq:Aczeta(f) def}, finite or infinite.
Equivalently,
\begin{equation*}
\widetilde{N_{\zeta}}(f) = N_{\zeta}(f) + \#\sing(f),
\end{equation*}
where
\begin{equation*}
\sing(f) = \left\{ x\in\Tb^{2}:\: f(x)=0,\, \nabla f(x) = 0 \right\}
\end{equation*}
is the set of singular nodal points of $f$.

Since by Bulinskaya's Lemma ~\cite[Proposition 6.12]{AW}, the singular set $\sing(f)$
is empty almost surely (that the statement of Bulinskaya's Lemma is valid in our concrete case
was established in ~\cite[Lemma 2.3]{ORW2007}), we have that
\begin{equation}
\label{eq:Nzeta=Ntild}
N_{\zeta}(f) = \widetilde{N_{\zeta}}(f) = \#\left\{x\in\Tb^{2}:\: f(x)=\langle\nabla f(x),\xi\rangle = 0  \right\}.
\end{equation}
That is, upon defining the Gaussian random field $G:\Tb^{2}\rightarrow\R^{2}$
\begin{equation}
\label{eq:G def}
G(x)=G_{\xi}(x) = (f(x),\langle \nabla f(x),\xi \rangle),
\end{equation}
then $N_{\zeta}$ equals almost surely the number of zeros of $G$. Let $J_{G}(x)$ be the Jacobian of $G$ given by
\begin{equation*}
\begin{split}
J_{G}(x) &= \det\left(\begin{matrix}
f_{1} &f_{2} \\ f_{11}\xi_{1}+f_{12}\xi_{2} & f_{12}\xi_{1}+f_{22}\xi_{2}
\end{matrix}\right) = f_{1}(f_{12}\xi_{1}+f_{22}\xi_{2}) - f_{2}(f_{11}\xi_{1}+f_{12}\xi_{2}),
\end{split}
\end{equation*}
where we denote $f_{i} = \partial f/\partial x_i$, $f_{ij} = \partial^2 f/\partial x_i\partial x_j$,
and all the derivatives of $f$ are evaluated at $x$.

The zero density function is
\begin{equation}
\label{eq:K1 phi*E}
K_{1}(x)=K_{1;\xi}(x) = \phi_{G(x)}(0,0)\cdot \E[|J_{G}(x)| \big| G(x)=0],
\end{equation}
where $\phi_{G(x)}$ is the probability density function of the random vector $G(x)\in \R^{2}$;
by the aforementioned stationarity \eqref{eq:stat law} of $f_{n}$, we have
$$K_{1}(x)\equiv K_{1}(0).$$
By Kac-Rice ~\cite[Theorem 6.3]{AW} and \eqref{eq:Nzeta=Ntild}, we have that
\begin{equation}
\label{eq:E[Nxi]=int(K1) Kac-Rice}
\E[N_{\zeta}] = \int\limits_{\Tb^{2}} K_{1}(x)dx,
\end{equation}
provided that the distribution of $G(x)$ is non-degenerate for every $x\in\Tb^{2}$. By stationarity, it is  sufficient to check non-degeneracy of  $G(0)$, which is valid since $(f(0),\nabla f(0)) \in \R^{3}$ is non-degenerate by the computation below.
The statement of Theorem \ref{thm:E[Nxi] d=2} follows upon substituting the statements of Lemma \ref{lem:phi(0,0)=1/sqrt(n)} and Proposition
\ref{prop:|JG(x)||G(x)=0} below into \eqref{eq:K1 phi*E} so that
\begin{equation*}
K_{1}(x) = \frac{1}{\sqrt{2}}n\cdot \left(1+\widehat{\mu_{n}}(4)\cdot \cos(4\theta)\right)^{1/2},
\end{equation*}
and then finally into \eqref{eq:E[Nxi]=int(K1) Kac-Rice}.
\end{proof}

In course of the proof of Theorem \ref{thm:E[Nxi] d=2} we used the following results established in
\S\ref{sec:2d density comp} below:

\begin{lemma}
\label{lem:phi(0,0)=1/sqrt(n)}
Let $G:\Tb^{2}\rightarrow \R^{2}$ be the Gaussian field defined by \eqref{eq:G def},
and $\phi_{G(x)}$ the probability density function of $G(x)$. Then for every $x\in \Tb^{2}$ we have
\begin{equation}
\label{eq:phi(0,0)=1/sqrt(n)}
\phi_{G(x)}(0,0) = \frac{1}{2\pi\sqrt{\det C_{G}(x)}} = \frac{1}{2^{3/2}\pi^{2}\sqrt{n}}.
\end{equation}

\end{lemma}

\begin{proposition}
\label{prop:|JG(x)||G(x)=0}
Let $G:\Tb^{2}\rightarrow \R^{2}$ be the Gaussian field defined by \eqref{eq:G def}, and $J_{G}(x)$ its Jacobian. Then
 the conditional expectation of $|J_G(x)|$ conditioned on $G=0$ is
\begin{equation*}
\E[|J_{G}(x)|\big|G(x)=0] = 2 \pi^{2} \left(1+\widehat{\mu_{n}}(4)\cdot \cos(4\theta)\right)^{1/2}\cdot n^{3/2}
\end{equation*}
\end{proposition}

\subsection{Proofs of Lemma \ref{lem:phi(0,0)=1/sqrt(n)} and Proposition \ref{prop:|JG(x)||G(x)=0}: evaluating the zero density}

\label{sec:2d density comp}

\begin{proof}[Proof of Lemma \ref{lem:phi(0,0)=1/sqrt(n)}]
  The covariance matrix of $(f(x),\nabla f(x))$ was computed in \cite[Proposition 4.1]{RW2008} to be
\begin{equation*}
C_{(f,\nabla f)} = \left(\begin{matrix}
1 \\ &2\pi^{2}nI_{2}
\end{matrix}\right),
\end{equation*}
in particular $f(x)$ is independent of $\nabla f(x)$;
hence the covariance matrix of $G$ is
\begin{equation*}
C_{G}(x) = \left(\begin{matrix}
1 \\&2\pi^{2}n
\end{matrix}\right),
\end{equation*}
where we used
$$\var(\langle \nabla f(x),\xi\rangle) = \xi_{1}^{2}\var(f_{1})+\xi_{2}^{2}\var(f_{2}) =
2\pi^{2}n,$$ since $$\xi_{1}^{1}+\xi_{2}^{2}=1.$$ Thus
\begin{equation*}
\phi_{G(x)}(0,0) = \frac{1}{2\pi\sqrt{\det C_{G}(x)}} = \frac{1}{2^{3/2}\pi^{2}\sqrt{n}}.
\end{equation*}

\end{proof}

\begin{proof}[Proof of Proposition \ref{prop:|JG(x)||G(x)=0}]

We are going to work under the assumption $\xi_{2}\ne 0$; one can easily see that the same result
holds for $\xi_{2} = 0$ true, e.g. by switching between $\xi_{1}$ and $\xi_{2}$; by stationarity we may assume $x=0$.
Since $f$ is a Laplace eigenfunction of eigenvalue $4\pi^{2}n$, we have that\footnote{This fact helps in simplifying the computation of the first intensity by allowing us to reduce the size of the covariance matrix, as seen in  another few steps.}
\begin{equation*}
f(x) = - \frac{1}{4\pi^{2}n}(f_{11}+f_{22}),
\end{equation*}
and therefore
\begin{equation*}
\E[|J_{G}(x)|\big| G(x)=0] = \E[|J_{G}(x)| \big| f_{11}+f_{22}=0,\,f_{1}\xi_{1}+f_{2}\xi_{2}=0].
\end{equation*}
Then (recall that we assumed $\xi_{2}\ne 0$)
\begin{equation*}
\begin{split}
&\E[|J_{G}(x)||G(x)=0] \\&= \E[|J_{G}(x)| \big| f_{11}+f_{22}=0,\,f_{1}\xi_{1}+f_{2}\xi_{2}=0]
\\&=
\E[\left| f_{1}(f_{12}\xi_{1}+f_{22}\xi_{2}) - f_{2}(f_{11}\xi_{1}+f_{12}\xi_{2})\right| \, \big| f_{11}+f_{22}=0,\, f_{1}\xi_{1}+f_{2}\xi_{2}=0]
\\&= \E[\left| f_{1}(f_{12}\xi_{1}-f_{11}\xi_{2}) - f_{2}(f_{11}\xi_{1}+f_{12}\xi_{2})\right| \, \big| f_{11}+f_{22}=0,\, f_{1}\xi_{1}+f_{2}\xi_{2}=0]
\\&=\E\left[\left| f_{1}\cdot \left( f_{12}\xi_{1}-f_{11}\xi_{2} + f_{11}\frac{\xi_{1}^{2}}{\xi_{2}}+f_{12}\xi_{1}   \right)\right| \big| f_{11}+f_{22}=0,\, f_{1}\xi_{1}+f_{2}\xi_{2}=0 \right]
\\&=
\E\left[\left| f_{1}\cdot \left( 2f_{12}\xi_{1}-f_{11}\frac{\xi_{1}^{2}-\xi_{2}^{2}}{\xi_{2}}\right)\right| \big| f_{11}+f_{22}=0,\, f_{1}\xi_{1}+f_{2}\xi_{2}=0 \right].
\end{split}
\end{equation*}
Hence we are interested in the distribution of $(f_{1},f_{11},f_{12})(x)$
conditioned on
$$f_{11}+f_{22}=0,\, f_{1}\xi_{1}+f_{2}\xi_{2}=0.$$
The covariance matrix of $(f_{1},f_{2},f_{11},f_{12},f_{22})$ is given by Lemma \ref{lem:covar deriv},
we will then compute the covariance matrix
of $(f_{1},f_{11},f_{12},f_{1}\xi_{1}+f_{2}\xi_{2},f_{11}+f_{22})$, and then condition on
the last two variables.
To avoid carrying on the constants we transform the variables
\begin{equation*}
((X_{1},X_{2}),(X_{3},X_{4},X_{5})) = \left(\frac{1}{(2\pi^{2})^{1/2}\sqrt{n}}(f_{1},f_{2}), \frac{1}{(2\pi^{4})^{1/2}n}(f_{11},f_{12},f_{22})\right)
\end{equation*}
with covariance matrix
\begin{equation*}
C_{X_{1},\ldots, X_{5}} =
\left( \begin{matrix}
I_{2\times 2} &
\\ 0 &A_{3\times 3}
\end{matrix}\right),
\end{equation*}
with
\begin{equation*}
A = \left( \begin{matrix}
3+\widehat{\mu_{n}}(4) &0 &1- \widehat{\mu_{n}}(4)   \\
0 &1- \widehat{\mu_{n}}(4) &0 \\
1- \widehat{\mu_{n}}(4) &0 &3+\widehat{\mu_{n}}(4),
\end{matrix}  \right)
\end{equation*}
and we are to compute
\begin{equation}
\label{eq:E[JG] terms X}
\begin{split}
&\E[|J_{G}(x)| \big| G(x)=0] =
\\2\pi^{3}n^{3/2} \cdot &E\left[\left| X_{1}\cdot \left( 2X_{4}\xi_{1}-X_{3}\frac{\xi_{1}^{2}-\xi_{2}^{2}}{\xi_{2}}\right)\right| \big| X_{3}+X_{5}=0,\, X_{1}\xi_{1}+X_{2}\xi_{2}=0 \right].
\end{split}
\end{equation}

Next we compute the covariance matrix of $(X_{1},X_{3},X_{4},X_{3}+X_{5},X_{1}\xi_{1}+X_{2}\xi_{2})$
to be
\begin{equation*}
C_{(X_{1},X_{3},X_{4},X_{3}+X_{5},X_{1}\xi_{1}+X_{2}\xi_{2})} =
\left(\begin{matrix}
B_{3\times 3} &D_{3\times 2} \\ D^{t}_{2\times 3} &E_{2\times 2}
\end{matrix} \right),
\end{equation*}
where $B=C_{X_{1},X_{3},X_{4}}$ is the covariance matrix of $(X_{1},X_{3},X_{4})$,
$E=C_{X_{3}+X_{5},X_{1}\xi_{1}+X_{2}\xi_{2}}$ is the covariance matrix of $$(X_{3}+X_{5},X_{1}\xi_{1}+X_{2}\xi_{2}),$$
and $$D=\E[(X_{1},X_{3},X_{4})^{t}(X_{3}+X_{5},X_{1}\xi_{1}+X_{2}\xi_{2})].$$
From the above it follows directly that
\begin{equation*}
B = \left(\begin{matrix}
1 & &\\ &3+\widehat{\mu_{n}}(4)\ \\ & &1- \widehat{\mu_{n}}(4)
\end{matrix}\right),
\end{equation*}
\begin{equation*}
D = \left(\begin{matrix}
0 &\xi_{1} \\ 4 & 0 \\ 0 &0
\end{matrix}\right),
\end{equation*}
\begin{equation*}
E=\left( \begin{matrix}
8 & 0 \\ 0 &1
\end{matrix} \right).
\end{equation*}

Let $Y=(Y_{1},Y_{2},Y_{3})$ be the vector $(X_{1},X_{3},X_{4})$ conditioned on
$$X_{3}+X_{5}=X_{1}\xi_{1}+X_{2}\xi_{2}=0,$$ so that under the new notation \eqref{eq:E[JG] terms X} is
\begin{equation}
\label{eq:EJG terms Y}
\begin{split}
&\E[|J_{G}(x)| \big| G(x)=0] \\&= 2\pi^{3}n^{3/2}\E\left[\left| Y_{1}\cdot \left( 2Y_{3}\xi_{1}-Y_{2}\frac{\xi_{1}^{2}-\xi_{2}^{2}}{\xi_{2}}\right)\right|\right].
\end{split}
\end{equation}
The covariance matrix of $Y$ is
\begin{equation*}
\begin{split}
C_{Y}=B-DE^{-1}D^{t} &= \left(\begin{matrix}
1-\xi_{1}^{2} \\&1+ \widehat{\mu_{n}}(4) \\ & &1- \widehat{\mu_{n}}(4)
\end{matrix} \right) \\&=  \left(\begin{matrix}
\xi_{2}^{2} \\ &1+ \widehat{\mu_{n}}(4) \\ & &1- \widehat{\mu_{n}}(4)
\end{matrix} \right),
\end{split}
\end{equation*}
where for the above we computed
\begin{equation*}
\begin{split}
DE^{-1}D^{t} &= \left(\begin{matrix}
0 &\xi_{1} \\ 4 & 0 \\ 0 &0
\end{matrix}\right) \cdot \left( \begin{matrix}
\frac{1}{8} \\ & 1
\end{matrix} \right) \cdot \left( \begin{matrix}
0 &4 &0 \\ \xi_{1} & 0 &0 \end{matrix}\right)
\\&= \left( \begin{matrix}
0 &\xi_{1} \\ \frac{1}{2} & 0 \\ 0 &0
\end{matrix} \right) \cdot \left( \begin{matrix}
0 &4 &0 \\ \xi_{1} & 0 &0\end{matrix}\right) = \left( \begin{matrix}
\xi_{1}^{2} \\ &2 \\ & & 0
\end{matrix} \right)
\end{split}
\end{equation*}

We may simplify the expression \eqref{eq:EJG terms Y} using the fact that the $Y_{j}$ are independent:
\begin{equation}
\label{eq:EJG terms Z}
\begin{split}
\E[|J_{G}(x)| \big| G(x)=0] &= 2\pi^{3}n^{3/2}\E[|Y_{1}|]
\cdot \E\left[\left|2Y_{3}\xi_{1}-Y_{2}\frac{\xi_{1}^{2}-\xi_{2}^{2}}{\xi_{2}}\right|\right]
\\&= 2^{1/2}\pi^{5/2}n^{3/2}\cdot
\E\left[\left| \left( 2Y_{3}\xi_{1}\xi_{2}-Y_{2}\left(\xi_{1}^{2}-\xi_{2}^{2}\right)\right)\right|\right]
\\&= 2^{1/2}\pi^{5/2}n^{3/2}\cdot \E\left[ |Z_{1}\sin(2\theta)+Z_{2}\cos(2\theta)   | \right],
\end{split}
\end{equation}
where $(Z_{1}=Y_{3},Z_{2}=-Y_{2})$ is a centered Gaussian with covariance
$$\left(\begin{matrix}
1-\widehat{\mu_{n}}(4) \\ &1+ \widehat{\mu_{n}}(4)
\end{matrix} \right)  ,$$ also valid for $\xi_{2}=0$
($\theta$ is the direction of $\zeta$, or of $\xi$ by the sign invariance of the distribution of $Z_{1}$, $Z_{2}$).

The random variable $$A:=Z_{1}\sin(2\theta)+Z_{2}\cos(2\theta)$$ is centered Gaussian, whose variance is
$$\var(A) = (1- \widehat{\mu_{n}}(4))\sin(2\theta)^{2} + (1+ \widehat{\mu_{n}}(4))\cos(2\theta)^{2} = 1+\widehat{\mu_{n}}(4)\cos(4\theta),$$
and \eqref{eq:EJG terms Z} is
\begin{equation*}
\begin{split}
\E[|J_{G}(x)| \big| G(x)=0] &= 2^{1/2}\pi^{5/2}n^{3/2} \cdot \E[|A|]\\& = 2^{1/2}\pi^{5/2}n^{3/2}\cdot \sqrt{\frac{2}{\pi}} \sqrt{\var(A)}
\\&= 2 \pi^{2}n^{3/2} \cdot (1+\widehat{\mu_{n}}(4)\cos(4\theta))^{1/2},
\end{split}
\end{equation*}
which is the statement of Proposition \ref{prop:|JG(x)||G(x)=0}.

\end{proof}

\subsection{Auxiliary lemmas}

\begin{lemma}[Cf. ~\cite{K-K-W}, Lemma 8.1]
\label{lem:sum lambda1^4,lambda1^2lambda2^2}
We have
\begin{equation*}
\begin{split}
\frac{1}{N}\sum\limits_{\lambda\in\Ec_{n}}\lambda_{1}^{4} = n^{2}\left(\frac{3}{8}+\frac{1}{8}\widehat{\mu_{n}}(4)\right),
\end{split}
\end{equation*}
and
\begin{equation*}
\begin{split}
\frac{1}{N}\sum\limits_{\lambda\in\Ec_{n}}\lambda_{1}^{2}\lambda_{2}^{2} =
\frac{n^{2}}{8}\left(1- \widehat{\mu_{n}}(4)  \right) .
\end{split}
\end{equation*}

\end{lemma}

\begin{lemma}
\label{lem:covar deriv}
Let $f=f_{n}$ be the arithmetic random waves (the random field \eqref{eq:f arith rand wav def} where $c_{\lambda}$ are
assumed to be i.i.d. standard Gaussian
save to \eqref{eq:c(-lam),c(lam)}), and
$X=(f_{1},f_{2},f_{11},f_{12},f_{22})$ vector of various derivatives evaluated at $x=0$. Then
$X$ is centered multivariate Gaussian with covariance matrix
\begin{equation*}
\begin{split}
&C_{f_{1},f_{2},f_{11},f_{12},f_{22}} \\&=
\left( \begin{matrix}
2\pi^{2}n &0 &0 &0 &0 \\
0 & 2\pi^{2}n &0 &0 &0 \\
0 &0 &2\pi^{4}n^{2}\left(3+\widehat{\mu_{n}}(4)\right) &0 &2\pi^{4} n^{2}\left(1- \widehat{\mu_{n}}(4)  \right) \\
0 &0 &0 &2\pi^{4}n^{2}\left(1- \widehat{\mu_{n}}(4)\right) &0 \\
0 &0 &2\pi^{4} n^{2}\left(1- \widehat{\mu_{n}}(4)\right) &0 &2\pi^{4}n^{2}\left(3+\widehat{\mu_{n}}(4)\right)
\end{matrix}  \right).
\end{split}
\end{equation*}

\end{lemma}

\begin{proof}
Recall that the covariance function of $f_{n}$ is given  by \eqref{eq:r covar def}.
We have, using the symmetries,
\begin{equation*}
\E[f_{1}(x)^{2}]=-r_{11}(0)=\E[f_{2}(x)^{2}] = 2\pi^{2}n
\end{equation*}
\begin{equation*}
\E[f_{1}(x)f_{2}(x)]=-r_{12}(0) = 0,
\end{equation*}
\begin{equation*}
\E[f_{1}(x)f_{11}(x)] =-r_{111}(0)= 0,
\end{equation*}
\begin{equation*}
\E[f_{1}(x)f_{12}(x)]=-r_{112}(0)=0,
\end{equation*}
\begin{equation*}
\begin{split}
\E[f_{11}(x)^{2}]&=\E[f_{22}(x)^{2}] =r_{1111}(0) = \frac{16\pi^{4}}{N}\sum\limits_{\lambda\in\Ec_{n}}\lambda_{1}^{4}
\\&= 16\pi^{4}n^{2}\left(\frac{3}{8}+\frac{1}{8}\widehat{\mu_{n}}(4)\right) = 2\pi^{4}n^{2} \left(3+\widehat{\mu_{n}}(4)\right)
\end{split}
\end{equation*}
by Lemma \ref{lem:sum lambda1^4,lambda1^2lambda2^2}, and
\begin{equation*}
\begin{split}
\E[f_{12}(x)^{2}]&= \E[f_{11}(x)f_{22}(x)] =
\frac{16\pi^{4}}{N}\sum\limits_{\lambda\in\Ec_{n}}\lambda_{1}^{2}\lambda_{2}^{2}
 = 2\pi^{4} n^{2}\left(1- \widehat{\mu_{n}}(4)  \right).
\end{split}
\end{equation*}

\end{proof}

\appendix

\section{Separable domains}\label{sec:separable}

We describe some  cases when the nodal sets, hence $N_\zeta(f)$, can be explicitly computed.

\subsection{Irrational rectangles}
Take a rectangle with width $ \pi/\sqrt{\alpha}$ and height $ \pi$, with aspect ratio $\sqrt{\alpha}  $, and assume that $\alpha$ is irrational.
Then the eigenvalues of the Dirichlet Laplacian consist  of the numbers $  \alpha m^2+n^2$ with integers $m,n\geq 1$, and the corresponding eigenfunctions are
$$f_{m,n}(x,y) =\sin(\sqrt{\alpha}mx)\sin(ny) \;.
$$
The nodal lines consist of a rectangular grid, and one has $N_\zeta(f_{m,n})=0\;{\rm or}\;\infty$.

\subsection{The unit disk}
Let $\Omega = \{|x|\leq 1\}$ be the unit disk, and $(r,\theta)$ be polar coordinates.
The eigenfunctions of the Dirichlet Laplacian are
$$ f_{m,k}(r,\theta) = J_m(j_{m,k}r)\cos(m\theta+\phi) $$
where $J_m(z)$ is the Bessel function, with zeros $\{j_{m,k}:k\geq 1\}$, and $\phi\in [0,2\pi)$ is arbitrary. The corresponding eigenvalue is
\begin{equation}
\label{eq:eigval disk}
E=j_{m,k}^2.
\end{equation}
In particular, for $m\geq 1$ the eigenspaces have dimension two.

We will need McCann's inequality \cite{McCann}
\begin{equation}\label{McCann ineq}
j_{m,k}^2 \geq \pi^2(k-\frac 14)^2 +m^2.
\end{equation}
For $m=0$ (the radial case), the eigenfunctions are $f_{0,k}(r,\theta) =J_0(j_{0,k}r)$, $0\leq r\leq 1$, and have $k-1$ interior nodal lines, which are the concentric circles $r=j_{0,\ell}/j_{0,k}$, $\ell=1,\dots, k-1$. Thus for any direction $\zeta\in S^1$, we have
$$ N_\zeta(f_{0,k}) = 2(k-1).$$

For $m\geq 1$, the nodal line of the eigenfunction $f_{m,k}$ is a union of the $m$ diameters $\cos(m\theta+\phi)=0$ and $k-1$ concentric circles $r=j_{m,\ell}/j_{m,k}$, $\ell=1,\dots,k-1$ (for $k=1$ there are only diameters), see Figure~\ref{fig:diskm3k5}. Thus there are $2m$ values of $\zeta$ where $N_\zeta(f_{m,k})=\infty$, and for all other directions we have
$$N_\zeta(f_{m,k})=2(k-1).$$
Using
McCann's inequality \eqref{McCann ineq}, and \eqref{eq:eigval disk}, the above yields that for $m\geq 0$, $k\geq 1$, except for $2m\leq 2\sqrt{E}$ directions where $N_\zeta(f_{m,k})=\infty$, we have
$$N_\zeta(f_{m,k}) \leq \frac 2\pi \sqrt{E}  \;.
$$

\begin{figure}[ht]
\begin{center}
\includegraphics[height=60mm]{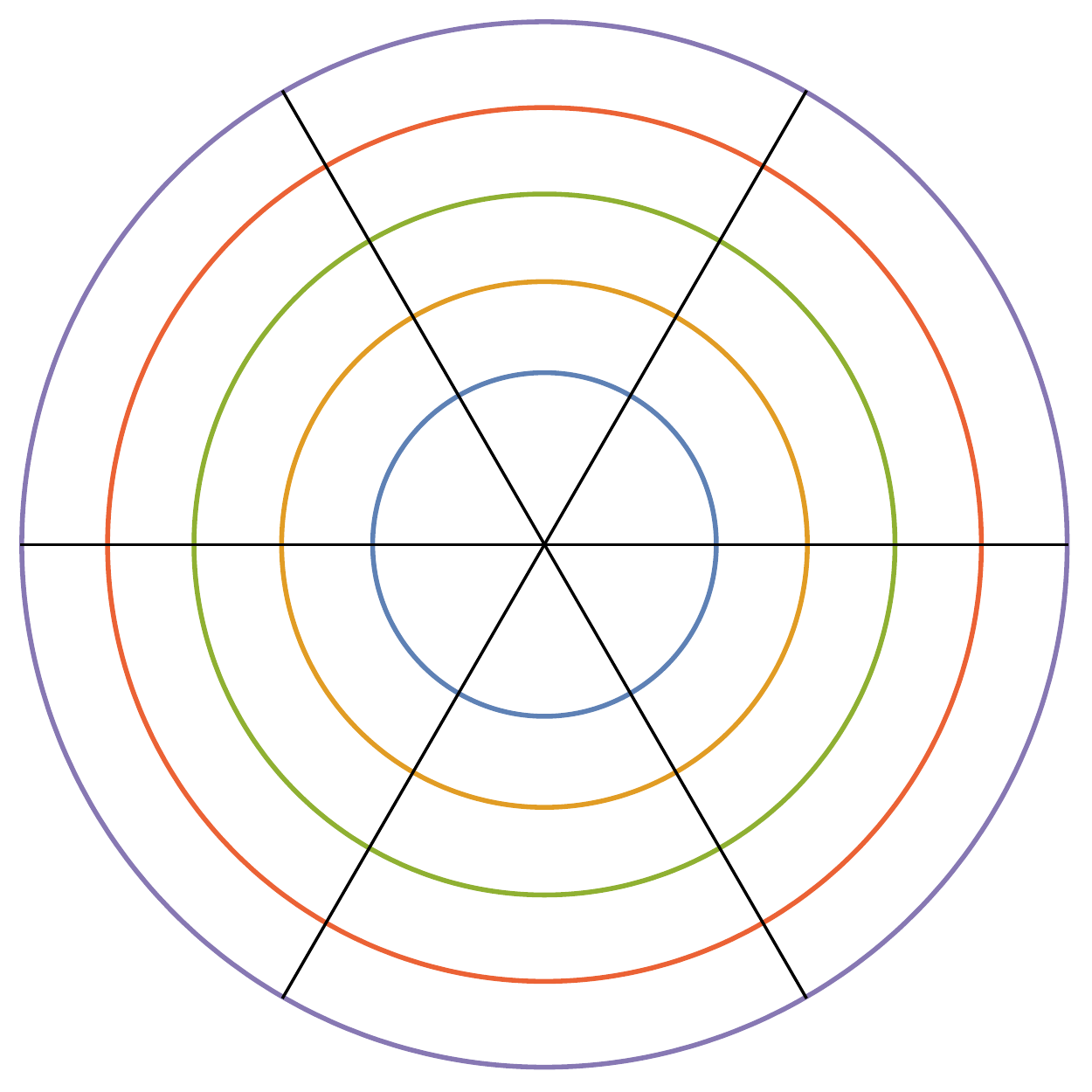}
\caption{The nodal line of the disk eigenfunction $f_{3,5}(x) = J_3(j_{3,5}r)\cos(3\theta)$, which consists of $3$ diameters and $4$ circles.}
\label{fig:diskm3k5}
\end{center}
\end{figure}

\end{document}